\documentclass[11pt,reqno]{amsart}

\usepackage[T1]{fontenc}
\usepackage{graphicx}
\usepackage{cite,enumerate}

\usepackage{color}
\definecolor{MyLinkColor}{rgb}{0,0,0.4}

\newcommand{\tr}{\mathop{\rm tr}\nolimits}
\newcommand{\p}{\partial}

\newcommand{\h}{\rho}

\newcommand{\ov}{\overline}
\newcommand{\wh}{\widehat}
\newcommand{\A}{\mathcal{A}}
\newcommand{\B}{\mathcal{B}}

\newcommand{\0}{\Omega}

\newcommand{\V}{\mathcal{V}}

\newcommand{\kL}{\mathcal{L}}

\newcommand{\D}{\mathcal{D}}

\newcommand{\s}{\mathbb{S}}

\newcommand{\R}{\mathbb{R}}

\newcommand{\C}{\mathbb{C}}
\newcommand{\cC}{\mathcal{C}}
\newcommand{\cO}{\mathcal{O}}
\newcommand{\x}{\mathcal{T}}
\newcommand{\mS}{\mathcal{S}}
\newcommand{\N}{\mathbb{N}}

\newcommand{\Z}{\mathbb{Z}}

\newtheorem{thm}{Theorem}[section]

\newtheorem{lem}[thm]{Lemma}

\theoremstyle{remark}

\numberwithin{equation}{section}

\title[Necrotic tumor growth]{Analysis of a mathematical model describing  necrotic tumor growth}


\usepackage[colorlinks=true,linkcolor=MyLinkColor,citecolor=MyLinkColor]{hyperref}

\author[J. Escher]{Joachim Escher}
\address{Institut f{\"u}r Angewandte Mathematik, Leibniz Universit{\"a}t Hannover, Welfengarten~1, 30167 Hannover, Germany. }
\email{escher@ifam.uni-hannover.de}

\author[A.-V. Matioc]{Anca-Voichita Matioc}
\address{Institut f{\"u}r Angewandte Mathematik, Leibniz Universit{\"a}t Hannover, Welfengarten~1, 30167 Hannover, Germany. }
\email{matioca@ifam.uni-hannover.de}

\author[B.-V. Matioc]{Bogdan-Vasile Matioc}
\address{Institut f{\"u}r Angewandte Mathematik, Leibniz Universit{\"a}t Hannover, Welfengarten~1, 30167 Hannover, Germany. }
\email{matioca@ifam.uni-hannover.de}

\usepackage[colorlinks=true,linkcolor=MyLinkColor,citecolor=MyLinkColor]{hyperref} 

\begin{document}

\begin{abstract}
In this paper we study a model describing the  growth of necrotic tumors in different regimes of vascularisation.
The tumor consists of a necrotic core  of death cells and a surrounding nonnecrotic shell.
The corresponding mathematical formulation is a moving boundary problem where both boundaries delimiting the nonnecrotic shell are allowed  to evolve in time.
We determine all  radially symmetric stationary solutions of the problem  and reduce the moving boundary problem into a nonlinear evolution. Parabolic theory  provides us the perfect context in order to show local well-posed of the problem for small initial data.
\end{abstract}

\maketitle

\section {The mathematical model}
In this paper we study a moving bondary problem describing the growth of a necrotic tumor  in the absence of inhibitor.
The model  purposed initially in  \cite{BC,  FR, Gr}, was reformulated by using algebraic manipulations \cite{VCris, EM} 
to describe evolution of tumors in all regimes of vascularisation.
Nevertheless, the  analysis in \cite{VCris, EM} is simplified by the assumption  that the tumor core is nonnecrotic.
This aspect is considered into our modeling, where following \cite{Gr, Gr2, CF}, we assume that the  tumor consists of a core of death cells (necrotic core) and a shell of life-proliferating cells surrounding 
the core (nonnecrotic shell).
The blood supply provides the nonnecrotic region with nutrients, while there is no blood supply in the necrotic region and the concentration of nutrients is at a constant level which cannot sustain cell proliferation.
However the model presented here  includes two moving boundaries, 
one parametrising the boundary of the necrotic core and one for the outer boundary of the tumor, both of them having infinitely many degrees of freedom.
This fact makes the problem more involved in comparison to other models which either neglect the necrotic core \cite{Cui, CE, CE1, CEZ, CF, BF05} or 
consider only the radially symmetric problem when the tumors are annular domains \cite{Gr, Gr2, CF}.

The mathematical model is given by the following system of equations
\begin{equation}\label{eq:problem}
\left \{
\begin{array}{rlllll}
\Delta \psi &=& \psi   &\text{in} \ \Omega(t) ,&t\geq0, \\[1ex]
\Delta p &=& 0 & \text{in}\ \Omega(t),&t\geq0,  \\[1ex]
\psi &=& G & \text{on}\ \Gamma_1(t),&t\geq0,  \\[1ex]
\psi &=& \displaystyle G-\psi_0& \text{on}\ \Gamma_2(t),&t\geq0,  \\[2ex]
p&=& \kappa_{\Gamma_1(t)}- AG \displaystyle\frac{|x|^2}{4} &\text{on} \ \Gamma_1(t),&t\geq0,\\[2ex]
p&=& \kappa_{\Gamma_2(t)}-AG \displaystyle \frac{|x|^2 }{4}-\psi_0 &\text{on} \ \Gamma_2(t),&t\geq0,\\[2ex]
V_i(t)&=&\displaystyle\p_{\nu_i}\psi -\p_{\nu_i}p -AG \displaystyle\frac{\nu_i\cdot x}{2} 
 & \text{on}\ \Gamma_i(t),&t>0, \,i=1,2 \\[2ex]
\Omega(0)&=&\Omega_0,&&
\end{array}
\right.
\end{equation}
where $\0(t)\subset \R^2$ is the domain occupied by the nonnecrotic shell, $\psi$ is the rate at which nutrient is added to $\0(t)$, over   the outer boundary $\Gamma_1(t)$, by the vascularisation,
$p$ is the pressure, $\Gamma_2(t)$ is the interior boundary enclosing the necrotic core,
$\nu_i$ the outward orientated normal  and $\kappa_{\Gamma_i(t)}$ the curvature of $\Gamma_i(t)$, $i=1,2.$   
By convention, $\kappa_{\Gamma_1(t)}$ is positive and $\kappa_{\Gamma_2(t)}$ negative if $\Gamma_i(t)$, $1\leq i\leq 2$ are close to a circle.
 Moreover, $V_i(t)$ stands for the normal velocity of $\Gamma_i(t),$ the constants $A, G\in\R $ have biological relevancy being related to cell proliferation, cell apoptosis, and vascularisation.
The scalar $\psi_0>0$ is linked with the constant nutrient concentration assumed within the necrotic region. 
The initial tumor domain is given by $\0_0.$

For a precise deduction of the system \eqref{eq:problem} and its biological meaning we refer to \cite{VCris, EM}, the only difference to the model presented there being  the consideration of the interior necrotic
 region bounded by $\Gamma_2(t)$. 
 
 The first main result of this paper is the following theorem:
 
\begin{thm}\label{T:1} Given $(R_1,R_2)\in(0,\infty)^2$ with $R_2<R_1$, let $\psi_0^c$ be the constant defined by \eqref{eq:const}. 
There exists $A\in\R$ and  $G\in\R\setminus\{0\},$ 
such that the annulus 
\[
A(R_1,R_2):=\{x\in\R^2\,:\,R_2<|x|<R_1\},
\] 
 is a stationary solution of problem \eqref{eq:problem} provided $\psi_0\neq\psi_0^c.$
 Moreover, $A$ and $G$ are uniquelly determined by $R_1, R_2,$ and $\psi_0.$
 
If $G=0,$ then problem \eqref{eq:problem} has no radially symmetric stationary solutions.
\end{thm} 
 
 In contrast to \cite{VCris, EM}, where the radially symmetric stationary tumors are circles with radius which depends only on the constant $A$, the radii of the stationary annular tumors
 found in Theorem \ref{T:1} depend on both constants $A$ and $G$, but also on $\psi_0,$ cf.  \eqref{solAG}.

 In order to prove local-well-posedness of the moving boundary problem \eqref{eq:problem} the well-posedness of system \eqref{eq:problem} (see Theorem \ref{T:2} below)
 we introduce first a parametrisation for the  interfaces $\Gamma_1(t)$ and $\Gamma_2(t),$ which are the main unknowns of system \eqref{eq:problem}. 
 \begin{figure}\label{F:bifu} 
$$\includegraphics[width=0.5\linewidth]{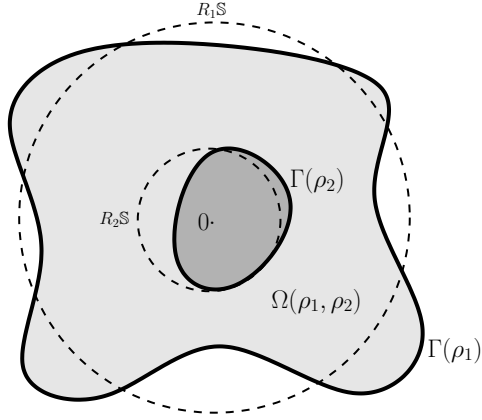}$$
\caption{Parametrisation of the tumor domain}
\end{figure}
 Let $0<R_2<R_1$ be given and fix $\alpha\in(0,1).$
 We set
\[
\V:=\{\h\in h^{4+\alpha}(\s)\,:\, \|\h\|_{C(\s)}<a\},
\]  
where
\[
a<\frac{R_1-R_2}{R_1+R_2}.
\]
The small H\"older space  $h^{m+\beta}(\s),$ $\beta\in(0,1) $ and $m\in\N,$ is defined as the completion of the smooth functions in $C^{m+\beta}(\s).$ 
Each pair $(\h_1, \h_2)\in\V^2 $ parametrises a  $C^{4+\alpha} $-domain
\[
\0(\h_1,\h_2):=\left\{y\in\R^2\, :\, R_2(1+\h_2\left(y/|y|\right))<|y|<R_1(1+\h_1\left(y/|y|\right))\right\}.
\]
The condition on $a $ ensures that the boundary portions of $\Omega(\h_1,\h_2)$
\[
\Gamma(\h_i):=\{x\,:\,|x|=R_i(1+\h_i(x/|x|))\},
\]
$i=1,2,$ are disjoint (see Figure 1) for any choice of $(\h_1, \h_2)\in\V^2.$
 They  can be  seen to be  zero level sets, $\Gamma(\h_i)={N_{\h_i}}^{-1}(0)$, where
 $N_{\h_i}:\R^2\setminus\{0\}\to\R, i=1,2,$ are defined by 
\[
N_{\h_i}(x)=|x|-R_i-R_i\h_i\left(x/|x|\right),\qquad x\neq0.
\]
 Hence, the outward unit normal at $\p\Omega(\h_1,\h_2)$ is  given by 
 \[
\text{$\nu_{\h_1}=\frac{\nabla N_{\h_1}}{|\nabla N_{\h_1}|}$ on $\Gamma(\h_1)$,   and  $\nu_{\h_2}=-\frac{\nabla N_{\h_2}}{|\nabla N_{\h_2}|}$ on $\Gamma(\h_2)$.}
\]
If $(\h_1,\h_2):[0,T]\to\V^2$ describe the motion of the tumor, then we can  express the normal velocity of both boundary components   in terms of  $\h_i$ by the formula
\[
\text{$V_1(t)=-\frac{\p_t N_{\h_1}}{|\nabla N_{\h_1}|}$ on $\Gamma(\h_1(t))$, and $V_2(t)=\frac{\p_t N_{\h_2}}{|\nabla N_{\h_2}|}$ on $\Gamma(\h_2(t))$}.
\]

With this notation, system \eqref{eq:problem} becomes a problem with $\h_1$ and $\h_2$ as unknowns:
\begin{equation}\label{eq:P}
\left \{
\begin{array}{rlllll}
\Delta \psi &=& \psi   &\text{in} \ \Omega(\h_1,\h_2) ,&t\geq0, \\[1ex]
\Delta p &=& 0 & \text{in}\ \Omega(\h_1,\h_2),&t\geq0,  \\[1ex]
\psi &=& G & \text{on}\ \Gamma(\h_1),&t\geq0,  \\[1ex]
\psi &=& G-\psi_0& \text{on}\ \Gamma(\h_2),&t\geq0,  \\[1ex]
p&=& \kappa_{\Gamma(\h_1)}- AG \displaystyle\frac{|x|^2}{4} &\text{on} \ \Gamma(\h_1),&t\geq0,\\[1ex]
p&=& \kappa_{\Gamma(\h_2)}-AG \displaystyle \frac{|x|^2 }{4}-\psi_0 &\text{on} \ \Gamma(\h_2),&t\geq0,\\[1ex]
\p_tN_{\h_i}&=&-\langle \left.\nabla\psi -\nabla p -AG \displaystyle\frac{ x}{2}\right|\nabla N_{\h_i}\rangle 
 & \text{on}\ \Gamma(\h_i),&t>0, \,i=1,2, \\[1ex]
\h_1(0)&=&\h_{01},&&\\[1ex]
\h_2(0)&=&\h_{02}.&&
\end{array}
\right.
\end{equation}
A pair $(\h_1,\h_2,\psi,p)$ is called classical solution of \eqref{eq:problem} on $[0,T], T>0,$ if
\begin{align*}
\h_i\in C([0,T],\V)\cap C^1([0,T], h^{1+\alpha}(\s)), \, i=1,2,\\[1ex]
\psi(t,\cdot), p(t,\cdot)\in \mbox{\it buc}^{2+\alpha}(\0(\h_1(t),\h_2(t))), \, t\in[0,T],  
\end{align*}
and if $(\h_1,\h_2,\psi,p)$ solves \eqref{eq:P} pointwise.
Given $U\subset\R^2$ open, we set $\mbox{\it buc}^{2+\alpha}(U)$ to be the closure of the the smooth functions with bounded and uniformly continuous derivatives
 $\mbox{\it BUC}\,^{\infty}(U)$ within  $\mbox{\it BUC}\,^{2+\alpha}(U)$ (if $U$ is also bounded then $\mbox{\it BUC}\,^{2+\alpha}(U)=C^{2+\alpha}(\ov U).$)

 Concerning well-posedness of system \eqref{eq:problem}, our second main result states that problem \eqref{eq:problem}
 possesses a unique solution provided that initially the tumor is close to an annulus (which is not necessarily a stationary solution).
 
 \begin{thm}[Local well-posedness]\label{T:2} Let $0<R_2<R_1$, and $(A,G, \psi_0)\in\R^3$ be given.
 
 There exists an open neighbourhood $\cO\subset\V$ such that for all $(\h_1,\h_2)\in\cO^2,$
 problem \eqref{eq:P} possesses a unique classical solution defined on a maximal time interval $[0,T(\h_{01},\h_{02}))$ and which satisfies
 $(\h_1,\h_2)(t)\in\cO^2$ for all  $t\in[0,T(\h_{01},\h_{02}))$.
 \end{thm}

 The outline of this paper is as follows: we study in Section 2 the radially symmetric free boundary problem  which describes the stationary solutions of problem \eqref{eq:problem}
 and prove Theorem \ref{T:1}.
 In the last section of this paper we prove the local well-posedness result Theorem \ref{T:2}.

 \section{Radially symmetric stationary solutions} 
We determine in this section the radially symmetric steady-state solutions  of \eqref{eq:problem}, 
situation  when  the nonnecrotic shell is a steady annulus.

The most simple situation is the case  $G=0$, when the problem is invariant under translation and rotations..
Then,   the annulus $A(R_1,R_2)$  centred in zero with radii $R_1>R_2,$ is a stationary solution of  system \eqref{eq:problem} if and only if
\[
p'(R_i)=\psi'(R_i),\qquad i=1,2,
\]
 where  $p$ is the solution of the problem
 \begin{equation}\label{eq:p}
 \left \{
\begin{array}{rlllll}
 p''+\displaystyle\frac{1}{r}p' &=& 0, &  R_2<r<R_1,  \\[1ex]
p(R_1)&=& {R_1}^{-1} -AGR_1^2/4,&  \\[1ex]
p(R_2)&=&-{R_2}^{-1}-AGR_2^2/4-\psi_0,
\end{array}
\right.
 \end{equation}
 when $G=0.$
 System \eqref{eq:p} corresponds to the Dirichlet problem for the pressure $p$ in \eqref{eq:problem} (the second, fifth,  and sixth equations of \eqref{eq:problem}), where we used 
polar coordinates when expressing the Laplacian.
Notice that the boundary data are constants, thus $p$ depends  only on   $r$, the distance to the origin.

Given $G\in\R,$ the solution of \eqref{eq:p} is given by the relation 
$p(r)=a_{R_1R_2}\ln(r)+b_{R_1R_2},$ $r_2\leq |r|\leq R_1,$ with
 \begin{align*}
 a_{R_1R_2}&=\frac{{R_1}^{-1}+{R_2}^{-1}+AG\left(R_2^2-R_1^2\right)/4+\psi_0}{\ln(R_1/R_2)},\\[1ex]
 b_{R_1R_2}&={R_1}^{-1}-AGR_1^2/4-a_{R_1R_2}\ln(R_1).
 \end{align*}
Furthermore, $\psi$ is the  solution of  the problem
\begin{equation}\label{eq:psI}
 \left \{
\begin{array}{rlllll}
 \psi''+\displaystyle\frac{1}{r}\psi' -\psi&=& 0, &  R_2<r<R_1,  \\[1ex]
\psi(R_1)&=& G ,&  \\[1ex]
\psi(R_2)&=&G-\psi_0,
\end{array}
\right.
 \end{equation}
 when $G=0$.
Also, for fixed $G\in\R,$ the solution of \eqref{eq:psI} can be written as linear combination of modified  Bessel functions of first and second kind 
$\psi=c^1_{R_1R_2} I_0+c^2_{R_1R_2}K_0, $ with scalars
\begin{align*}
&c^1_{R_1R_2}= \displaystyle\frac{GK_0(R_2)+(\psi_0-G)K_0(R_1)}{ I_0(R_1)K_0(R_2)-I_0(R_2)K_0(R_1)},\\[1ex]
& c^2_{R_1R_2}= \displaystyle\frac{-GI_0(R_2)-(\psi_0-G)I_0(R_1)}{ I_0(R_1)K_0(R_2)-I_0(R_2)K_0(R_1)}.
\end{align*}

Consequently, $A(R_1,R_2)$ is a steady-state solution of \eqref{eq:problem} when $G=0$ if and only if
\begin{align}\label{eq:G=0}
\frac{\displaystyle\frac{1}{R_1}+\frac{1}{R_2}+\psi_0}{\ln(R_1/R_2)}\frac{1}{R_i}&=\psi_0\displaystyle\frac{K_0(R_1)I_1(R_i)+I_0(R_1)K_1(R_i)}{ I_0(R_1)K_0(R_2)-I_0(R_2)K_0(R_1)},\qquad i=1,2,
\end{align}
where we used the relations $I_0'=I_1$ and $K_0'=-K_1.$
It follows then easily that the system consisting of the equations \eqref{eq:G=0}  has solutions $(R_1,R_2)$ with $R_1>R_2$  exactly when 
\begin{equation}\label{eq:bine}
\frac{R_2}{R_1}=\frac{K_0(R_1)I_1(R_1)+I_0(R_1)K_1(R_1)}{K_0(R_1)I_1(R_2)+I_0(R_1)K_1(R_2)}.
\end{equation}
Equation \eqref{eq:bine} is obtained by expressing $\psi_0$ in both relations \eqref{eq:G=0} and setting them to be equal.
We show now that equality holds in the relation above only  when $R_1=R_2.$
Indeed, fix $R_1>0$ and consider the auxiliary  function $g:(0,R_1]\to\R$ with
\[
g(x)=K_0(R_1)xI_1(x)+I_0(R_1)xK_1(x)-R_1(K_0(R_1)I_1(R_1)+I_0(R_1)K_1(R_1))
\]
 for $0<x\leq R_1.$ 
Obviously $g(R_1)=0.$
If we show that the derivative  $g'$ has constant sign on $(0,R_1]$ then we are done, that is there is no positive  $R_2<R_1$ such that $(R_1,R_2) $ solves \eqref{eq:bine}.
Well-known properties of the modified Bessel functions (see \cite{AW}) lead to
\begin{align*}
g'(x)=&K_0(R_1)I_1(x)+I_0(R_1)K_1(x)\\[1ex]
&+K_0(R_1)x(I_0(x)-(1/x)I_1(x))+I_0(R_1)x(-K_0(x)-(1/x)K_1(x))\\[1ex]
=&x(I_0(x)K_0(R_1)-I_0(R_1)K_0(x))<0
\end{align*} 
for all $x\in(0,R_1).$
That the last expression is negative is a consequence of the following facts: $I_0$ and $K_0$ are both positive functions, $I_0$ is strictly increasing, and $K_0$ is strictly decreasing. 
Hence, problem \eqref{eq:problem} has no radially symmetric  stationary when $G=0$.

 Let now $G\neq0.$  
In this case  $A(R_1,R_2) $ is a steady-state solution of \eqref{eq:problem} exactly when
\begin{equation}\label{eq:K}
\psi_{R_1R_2}'(R_i)-p'(R_i)-AG\frac{R_i}{2}=0, \qquad i=1,2.
\end{equation}  
 Using again the relations  $I_0'=I_1$ and $K_0'=-K_1$, the identities \eqref{eq:K} re-write
\begin{equation*}
c_{R_1R_2}^1I_1(R_i)-c^2_{R_1R_2}K_1(R_i)-a_{R_1R_2}\frac{1}{R_i}-AG\frac{R_i}{2}=0, \qquad i=1,2,
\end{equation*}  
 which seem to be very involved as expressions of variables $R_1$ and $R_2$ when trying to solve the system consisting of both of them.
 However, they can be viewed as equations for $A$ and $G$ 
 \begin{equation}\label{eq:AG}
 a_iG+b_iAG=c_i,\qquad i=1,2,
 \end{equation}
 with coefficients $a_i, b_i$, and $c_i$  given by:
 \begin{align*}
 a_i&:=\frac{(K_0(R_2)-K_0(R_1))I_1(R_i)-(I_0(R_1)-I_0(R_2))K_1(R_i)}{I_0(R_1)K_0(R_2)-I_0(R_2)K_0(R_1)},\\[1ex]
 b_i&:=\frac{R_1^2-R_2^2}{4\ln(R_1/R_2)}\frac{1}{R_i}-\frac{R_i}{2},\\[1ex]
 c_i&:=-\psi_0\frac{K_0(R_1)I_1(R_i)+I_0(R_1)K_1(R_i)}{I_0(R_1)K_0(R_2)-I_0(R_2)K_0(R_1)}+\frac{\displaystyle{R_1}^{-1}+{R_2}^{-1}+\psi_0}{\ln(R_1/R_2)}\frac{1}{R_i}.
 \end{align*}
 The system of equations \eqref{eq:AG} has a (unique) solution $(A,G)$ with $G\neq0$ provided that
 \begin{equation}\label{eq:SyS}
\begin{aligned}
& a_1 b_2- a_2b_1 \not = 0,\qquad c_1 b_2- c_2 b_1\not =0,\\
&\text{and} \quad  c_1\not =0 \quad \text{or} \quad c_2 \not =0.
\end{aligned}
\end{equation}
\begin{figure}
$$\includegraphics[width=0.72\linewidth]{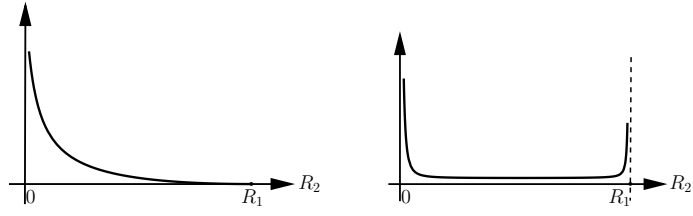}$$
\caption{The expression $a_1b_2-a_2b_1$ and $\psi_0^c$, for fixed $R_1$, as a function of the variable $R_2\in(0,R_1).$}
\end{figure}
The computation done for the case $G=0$ shows that $c_1$ and $c_2$ cannot be simultaneously zero when $R_2<R_1.$
For fixed $R_1>0$ we may see the expression 
$a_1b_2-a_2b_1$ as a function of $R_2\in(0,R_1).$
This function is decreasing with respect to $R_2$ (see Figure 2), thus $a_1b_2=a_2b_1$ only when $R_1=R_2.$
Furthermore, $b_1c_2=b_2c_1$ if and only if $\psi_0=\psi_0^c,$ where
\begin{equation}\label{eq:const}
\psi_0^c:=\frac{\left(b_1/R_1-b_2/R_2\right)\frac{1/R_1+1/R_2}{\ln(R_1/R_2)}}
{\frac{K_0(R_1)(b_1I_1(R_2)-b_2I_1(R_1))+I_0(R_1)(b_1K_1(R_2)-b_2K_1(R_1))}{I_0(R_1)K_0(R_2)-I_0(R_2)K_0(R_1)}+\frac{R_1^2-R_2^2}{2R_1R_2\ln(R_1/R_2)}}.
\end{equation}
It is not difficult to see that the numerator of the fraction is  negative, and the same holds true for the denominator, implying that $\psi_0^c>0.$
We plotted  in Figure 2 the expression  on the right hand side of \eqref{eq:const} for fixed $R_1>0$ in dependence of $R_2\in(0,R_1).$
Consequently, $A(R_1,R_2)$ is a stationary solution of \eqref{eq:problem} if and only if $\psi_0$ is not the critical constant given by \eqref{eq:const} and 
\begin{equation}\label{solAG}
A= \displaystyle\frac{a_1c_2-a_2c_1}{c_1b_2-c_2b_1}, \qquad G= \displaystyle\frac{c_1b_2-c_2b_1}{a_1b_2-a_2b_1}.
\end{equation}
This proves Theorem \ref{T:1}.

\section{The moving boundary problem}
This last section is dedicated entirely to the proof of our secon main result, Theorem \ref{T:2}.
In order to prove well-posedness of problem \eqref{eq:problem} in the context defined in the introduction
we transform first \eqref{eq:P} into a problem on the fixed domain $\Omega:=\Omega(0,0),$
with boundary $\Gamma_1:=R_1\s$ and $\Gamma_2:=R_2\s.$
This transformation will allow us to introduce solution operators related to  problem \eqref{eq:P} and which will enable us to 
reduce system \eqref{eq:P} into an abstract nonlinear evolution equation for the pair $(\h_1,\h_2).$

Pick therefore $0<R_2<R_1$, $(A,G,\psi_0)\in\R^3$, and $\alpha\in(0, 1)$.
Given $(\h_1,\h_2)\in\V^2,$ we define the mapping
$\Theta_{\h_1,\h_2}:\Omega\to\Omega(\h_1,\h_2)$
by the relation 
\[
\Theta_{\h_1,\h_2}(x)=\frac{(R_1-|x|)R_2(1+\h_2(x/|x|))+(|x|-R_2)R_1(1+\h_1(x/|x|))}{R_1-R_2}\frac{x}{|x|}
\]
for $x\in\0.$
One can easily check that $\Theta_{\h_1,\h_2}$ is an diffeomorphism, i.e. $\Theta_{\h_1,\h_2}\in\mbox{\it Diff}\,^{4+\alpha}(\0,\0(\h_1,\h_2)),$
which maps $\Gamma_i$ onto $\Gamma(\h_i), i=1,2.$ 
Using this diffeomorphism, we define the transformed operators
\[
\A(\h_1,\h_2):\mbox{\it buc}^{2+\alpha}(\0)\to \mbox{\it buc}^{\alpha}(\0),\quad \A(\h_1,\h_2)v:=\Delta\left(v\circ\Theta_{\h_1,\h_2}^{-1}\right)\circ \Theta_{\h_1,\h_2}-v,
\]
which is an elliptic operator depending analytically on $(\h_1,\h_2),$ i.e. 
\begin{equation}\label{eq:anA}
\A\in C^\omega(\V^2,\kL(\mbox{\it buc}^{2+\alpha}(\0),\mbox{\it buc}^{\alpha}(\0))),
\end{equation}
respectively the trace operators $\B_i:\V^2\times(\mbox{\it buc}^{2+\alpha}(\0))^2\to h^{1+\alpha}(\s)$ by
\begin{align*}
\B_i(\h_1,\h_2,v,q)&:=\frac{1}{R_i}\cC_i(\h_1,\h_2)v-\frac{1}{R_i}\cC_i(\h_1,\h_2)q-\D_i(\h_1,\h_2).
\end{align*}
Given $(\h_1,\h_2)\in\V^2,$ the linear operators $\cC_i(\h_1,\h_2)\in\kL(\mbox{\it buc}^{2+\alpha}(\0)), h^{1+\alpha}(\s)),$ $i=1,2,$ 
are given by
\[
\cC_i(\h_1,\h_2)v(y):=  \langle \left.\nabla\left(v\circ\Theta_{\h_1,\h_2}^{-1}\right) \right|\nabla N_{\h_i}\rangle \circ\Theta_{\h_1,\h_2}(R_iy)
\]
for $v\in \mbox{\it buc}^{2+\alpha}(\0)$ and $y\in\s.$ 
Moreover,
\[
\D_i(\h_1,\h_2):=-\frac{AG}{R_i} \langle \left. \displaystyle\frac{ x}{2}\right|\nabla N_{\h_i}\rangle \circ\Theta_{\h_1,\h_2}(R_iy).
\]
The operators $\cC_i$ and $\D_i, i=1,2,$ depend  analytically on $(\h_1,\h_2) $ too,  
\begin{equation}\label{eq:anB}
\text{$\cC_i\in C^\omega(\V^2, \kL(\mbox{\it buc}^{2+\alpha}(\0)), h^{1+\alpha}(\s)))$ and $\D_i\in C^\omega(\V^2, h^{1+\alpha}(\s)).$ }
\end{equation}

Having defined this operators we may re-write now \eqref{eq:P} in an equivalent form.
Namely, if $(\h_1,\h_2,\psi,p) $ is a solution of \eqref{eq:P}, $v:=\psi\circ\Theta_{\h_1,\h_2} $, and  $q:=p\circ\Theta_{\h_1,\h_2}$, then the tupel $(\h_1,\h_2,v,q) $
solves the following system
\begin{equation}\label{eq:TP}
\left \{
\begin{array}{rlllll}
\A(\h_1,\h_2) v &=& 0   &\text{in} \ \Omega ,&t\geq0, \\[1ex]
\A(\h_1,\h_2) q &=& 0 & \text{in}\ \Omega,&t\geq0,  \\[1ex]
v &=& G & \text{on}\ \Gamma_1,&t\geq0,  \\[1ex]
v &=& G-\psi_0& \text{on}\ \Gamma_2,&t\geq0,  \\[1ex]
q&=& \frac{1}{R_1} \kappa(\h_1)-\frac{AGR_1^2}{4}(1+\h_1)^2 &\text{on} \ \Gamma_1,&t\geq0,\\[1ex]
q&=& -\frac{1}{R_2} \kappa(\h_2)-\frac{AGR_2^2}{4}(1+\h_2)^2 -\psi_0&\text{on} \ \Gamma_2,&t\geq0,\\[1ex]
\p_t\h_i&=&\B_i(\h_1,\h_2,v,q)& \text{on}\ \s,&t>0, \,i=1,2, \\[1ex]
\h_1(0)&=&\h_{01},&&\\[1ex]
\h_2(0)&=&\h_{02},&&
\end{array}
\right.
\end{equation}
where $\kappa:\V\to h^{2+\alpha}(\s)$ is defined by
\[
\kappa(\h):=\frac{(1+\h)^2+2\h'^2-(1+\h)\h''}{((1+\h)^2+\h'^2)^{3/2}},\qquad \h\in\V,
\]
and we identified  functions on $\Gamma_i$ with those on $\s$, $i=1,2$, via the diffeomorphisms  $[\s\ni y\mapsto R_i y\in\Gamma_i].$ 

Though the problem becomes more involved (the diffeomorphism introduces additionally nonlinearities), \eqref{eq:TP} has the advantage that the sets where 
the differential equations and the boundary conditions are imposed do not change with time.
It is convenient now to introduce  solution  operators  to Dirichlet problem closely related to system \eqref{eq:TP}. 
\begin{lem}\label{L:SO} Given $(\h_1,\h_2)\in\V^2,$ we let $\x(\h_1,\h_2)$, $\mS(\h_1,\h_2)\in \mbox{\it buc}^{2+\alpha}(\0)$ denote the unique solution of   
\begin{equation}\label{eq:DP1}
\left \{
\begin{array}{rlllll}
\A(\h_1,\h_2) v &=& v   &\text{in} \ \Omega,  \\[1ex]
v &=& G & \text{on}\ \Gamma_1,\\[1ex]
v &=& G-\psi_0& \text{on}\ \Gamma_2,
\end{array}
\right.
\end{equation}
and 
\begin{equation}\label{eq:DP2}
\left \{
\begin{array}{rlllll}
\A(\h_1,\h_2) q &=& 0    &\text{in} \ \Omega,  \\[1ex]
q &=&\frac{1}{R_1} \kappa(\h_1)-\frac{AGR_1^2}{4}(1+\h_1)^2 & \text{on}\ \Gamma_1,\\[1ex]
q &=&-\frac{1}{R_2} \kappa(\h_2)-\frac{AGR_2^2}{4}(1+\h_2)^2 -\psi_0& \text{on}\ \Gamma_2,
\end{array}
\right.
\end{equation}
respectively.
The operators $\x$ and $\mS$ depend analytically on $(\h_1,\h_2).$
\end{lem}
\begin{proof} Given $(\h_1,\h_2)\in\V^2,$ the Dirichlet problems \eqref{eq:DP1} and \eqref{eq:DP2}
are uniquelly solvable, cf. \cite[Theorem 6.14]{GT}.
Moreover, since $\A$ and $\kappa$ depend analytically on their variables we deduce that also $\x$ and $\mS$ do that.
We may take now into consideration that $\x$ and $\mS$ map both the smooth functions into $\mbox{\it BUC}^{2+\alpha}(\0)$
 and conclude that their range is contained in $\mbox{\it buc}^{2+\alpha}(\0).$
\end{proof}

With this definion  \eqref{eq:TP} reduces to the following evolution equation
\begin{equation}\label{eq:CP}
\p_tX=\Phi(X)\qquad X(0)=X_0,
\end{equation}
where $X:=(\h_1,\h_2),$ $X_0:=(\h_{01},\h_{02}),$ and $\Phi:=(\Phi_1,\Phi_2).$
The components of the nonlocal and nonlinear operator $\Phi$ are defined as follows
\begin{align*}
\Phi_i(\h_1,\h_2)&:=\B_i(\h_i, \x(\h_1,\h_2),\mS(\h_1,\h_2)),\qquad i=1,2.
\end{align*}
In order to prove well-posedness of problem \eqref{eq:CP}
it suffices to show that
\[
\p\Phi(0)=
\left[
\begin{array}{cc}
\p_{\h_1}\Phi_1(0)&\p_{\h_2}\Phi_1(0)\\[1ex]
\p_{\h_1}\Phi_2(0)&\p_{\h_2}\Phi_2(0)
\end{array}
\right]
\]
generates a strongly continuous and analytic semigroup.
The key role is played by the operator $\mS$ which depends on the highest order  derivatives of $\h_i, i=1,2.$
We have:
\begin{thm} The operator $\Phi$ is analytic, i.e. $\Phi\in C^\omega(\V^2, (h^{1+\alpha}(\s))^2).$
Given $\beta\in(0,1),$ the Fr\'echet derivative $\p\Phi(0)$, seen as an unbounded operator in $(h^{1+\beta}(\s))^2$  with domain $(h^{4+\beta}(\s))^2$ generates a strongly  continuous and analytic semigroup in $\kL((h^{1+\beta}(\s))^2),$ i.e.
\[
-\p\Phi(0)\in \mathcal{H}((h^{4+\beta}(\s))^2,(h^{1+\beta}(\s))^2).
\] 
\end{thm}
\begin{proof}
The regularity assumption follows directly from \eqref{eq:anA} and \eqref{eq:anB}.
Moreover, since the constant $\alpha$ fixed at the begining of this section was arbitrary, we may replace $\alpha$ by $\beta$ and all the assetions already established remain valid.

Let us now study the Fr\'echet derivative of $\Phi$ in $0$.
One can easily see that the highest order terms in $(\h_1,\h_2)$ of $\p\Phi(0,0)[(\h_1,\h_2)]$ are those 
obtained when differentiating the curvature operator.

Consider first $\p_{\h_1}\Phi_1(0)$.
Since $\p\kappa(0)[\h]=-\h''-\h$  for $\h\in\h^{4+\alpha}(\s),$  we may decompose  
\[
\p_{\h_1}\Phi_1(0)[\h_1]=A_{11}+B_{11},
\]
where $B_{11}$ is an operator of first order, i.e. $B_{11}\in\kL(h^{2+\beta}(\s), h^{1+\beta}(\s)),$
\[
A_{11}\h_1:=-\frac{1}{R_1^2}C_1(0)(\Delta,\tr_1,\tr_2)^{-1}(0,\h_1'',0),\quad \forall \h_1\in h^{4+\beta}(\s),
\]
and $\tr_i, i=1,2,$ is the trace operator with respect to $\Gamma_i.$
We determine now a Fourier expansion for the highest order term of  $\p_{\h_1}\Phi_1(0)[h_1].$
Given  $\h_1\in\h^{4+\beta}(\s),$ the function $w:=(\Delta,\tr_1,\tr_2)^{-1}(0,h_1'',0)$ is the solution of the Dirichlet problem
\begin{equation}\label{eq:DPP}
\left \{
\begin{array}{rlllll}
\Delta w &=& 0    &\text{in} \ \Omega,  \\[1ex]
w &=&\h_1'' & \text{on}\ \Gamma_1,\\[1ex]
w &=&0& \text{on}\ \Gamma_2.
\end{array}
\right.
\end{equation}
If we expand
\[
\text{$\h_1(y)=\sum_{m}\wh \h_1(m) y^m $ and $w(ry)=\sum_{m\in\Z} w_m(r)y^k $}
\]
for $y\in\s$ and $R_2<r<R_1$, we find that 
 $w_0=0$,  and $w_m$ solves, for $|m|\geq 1,$ the problem 
\begin{equation*}
\left \{
\begin{array}{rlllll}
\displaystyle w_m''+\frac{1}{r}w_m'-\frac{m^2}{r^2}w_m &=& 0 &R_2<r<R_1 \\[1ex]
w_m(R_1)&=&-m^2\wh\h_1(m) \\[1ex]
w_m(R_2) &=&0.
\end{array}
\right.
\end{equation*}
Hence
\[
w_m(r)=-\frac{R_2^{m}r^{-m}-R_2^{-m}r^m}{R_2^{m}R_1^{-m}-R_2^{-m}R_1^m}m^2\wh \h_1(m),
\]
and therewith
\begin{align*}
A_{11}\h_1(y)&=-\frac{1}{R_1^2}\cC_1(0)w(y)=-\frac{1}{R_1^2}\langle\nabla w(R_1y)|y\rangle=-\frac{1}{R_1^2}\frac{d}{dr}\left.\left(w(ry)\right)\right|_{r=R_1}\\[1ex]
&=-\frac{1}{R_1^3}\sum_{m\in\Z\setminus\{0\}}\frac{R_1^{|m|}R_2^{-|m|}+R_1^{-|m|}R_2^{|m|}}{R_1^{|m|}R_2^{-|m|}-R_1^{-|m|}R_2^{|m|}}|m|^3\wh \h_1(m)y^m.
\end{align*}

We  proceed similarly and write $\p_{\h_2}\Phi_1(0)=A_{12}+B_{12},$
where $B_{12}\in\kL(h^{2+\beta}(\s), h^{1+\beta}(\s))$ and 
\[
A_{12}\h_2:=-\frac{1}{R_1R_2}C_1(0)(\Delta,\tr_1,\tr_2)^{-1}(0,0,\h_2'')\quad \forall \h_2\in h^{4+\beta}(\s).
\]
Given $\h_2\in h^{4+\beta}(\s),$ the function $w:=(\Delta,\tr_1,\tr_2)^{-1}(0,0,\h_2'')$ is the solution of linear Dirichlet problem
\begin{equation}\label{eq:DP6}
\left \{
\begin{array}{rlllll}
\Delta w &=& 0    &\text{in} \ \Omega,  \\[1ex]
w &=&0 & \text{on}\ \Gamma_1,\\[1ex]
w &=&\h_2''& \text{on}\ \Gamma_2.
\end{array}
\right.
\end{equation}
A Fourier series ansatz as we did before yields that 
\[
w(ry)=-\sum_{m\in\Z\setminus\{0\}} \frac{R_1^{m}r^{-m}-R_1^{-m}r^m}{R_1^{m}R_2^{-m}-R_1^{-m}R_2^m}m^2\wh \h_2(m)y^m 
\]
for all $y\in\s$ and $R_2<r<R_1,$ provided that $\h_2=\sum_{m\in\Z}\wh \h_2(m) y^m .$
Whence,
\[
A_{12}\sum_{m\in\Z}\wh \h_2(m) y^m=-\frac{1}{R_1^2R_2}\sum_{m\in\Z\setminus\{0\}}\frac{2}{R_1^{|m|}R_2^{-|m|}-R_1^{-|m|}R_2^{|m|}}|m|^3\wh \h_2(m)y^m.
\]

We consider now the second component $\Phi_2$ and continue our computation following the same scheme.
The second diagonal element of the matrix $\p\Phi(0)$ may be also written as the sum $\p_{\h_2}\Phi_2(0)=A_{22}+B_{22},$
with 
$B_{22}\in\kL(h^{2+\beta}(\s), h^{1+\beta}(\s))$ and 
\[
A_{22}\h_2:=-\frac{1}{R_2^2}C_2(0)(\Delta,\tr_1,\tr_2)^{-1}(0,0,\h_2'')\quad \forall \h_2\in h^{4+\beta}(\s).
\]
Using once more the expansion for the solution of \eqref{eq:DP6}, we find out that
\[
A_{22}\sum_{m\in\Z}\wh \h_2(m) y^m=-\frac{1}{R_2^3}\sum_{m\in\Z\setminus\{0\}}\frac{R_1^{|m|}R_2^{-|m|}+R_1^{-|m|}R_2^{|m|}}{R_1^{|m|}R_2^{-|m|}-R_1^{-|m|}R_2^{|m|}}|m|^3\wh \h_2(m)y^m.
\]
for all $\h_2=\sum_{m\in\Z}\wh \h_2(m) y^m $ within $h^{4+\beta}(\s).$
Finally,  $\p_{\h_1}\Phi_2(0)=A_{21}+B_{21},$ where $B_{21}\in\kL(h^{2+\beta}(\s), h^{1+\beta}(\s))$
and 
\[
A_{21}\h_1:=\frac{1}{R_1R_2}C_2(0)(\Delta,\tr_1,\tr_2)^{-1}(0,\h_1'',0)\quad \forall \h_2\in h^{4+\beta}(\s).
\]
Since $(\Delta,\tr_1,\tr_2)^{-1}(0,\h_1'',0)$ is the solution of \eqref{eq:DPP}, we may use the expansion found at that point of the proof and obtain that
\[
A_{21}\sum_{m\in\Z}\wh \h_1(m) y^m=-\frac{1}{R_1R_2^2}\sum_{m\in\Z\setminus\{0\}}\frac{2}{R_1^{|m|}R_2^{-|m|}-R_1^{-|m|}R_2^{|m|}}|m|^3\wh \h_1(m)y^m
\]
for all functions $\h_1=\sum_{m\in\Z}\wh \h_1(m) y^m $ in $h^{4+\beta}(\s).$

Let us notice that the operators $A_{ij}, 1\leq i,j\leq2,$ found above are all Fourier multipliers, since they are of the form
\[
\sum_{m\in\Z}\wh \h_1(m) y^m\mapsto \sum_{m\in\Z}M_k\wh \h_1(m)y^m
\]
with  symbol $(M_k)_{k\in\Z}\subset\C$.
Using \cite[Theorem 3.4]{EM4}, which is a   theorem characterising multiplier operators between H\"older spaces by studying some generalised Marcinkiewicz conditions for the symbol of the operator,
we find out that
$-A_{ii}\in \mathcal{H}(h^{4+\beta}(\s), h^{1+\beta}(\s)), i=1,2,$ and  that $A_{12}, A_{21}\in\kL(h^{2+\beta}(\s)).$ 
This may be seen form the following relations
\[
\text{$\frac{2}{R_1^{|m|}R_2^{-|m|}-R_1^{-|m|}R_2^{|m|}}\to_{|m|\to\infty}0$ and $\frac{R_1^{|m|}R_2^{-|m|}+R_1^{-|m|}R_2^{|m|}}{R_1^{|m|}R_2^{-|m|}-R_1^{-|m|}R_2^{|m|}}\to_{|m|\to\infty}1.$}
\]
Since $h^{2+\beta}(\s)$ is an intermidiate space,
\[
h^{2+\beta}(\s)=(h^{1+\alpha}(\s), h^{4+\alpha}(\s))_{(1+\beta-\alpha)/3},
\]
where $(\cdot|\cdot)$ denotes the interpolation functor introduced by Da Prato and Grisvard \cite{DG},
 we obtain form \cite[Proposition 2.4.1]{L} that the elements on the diagonal of $\p\Phi(0)$ generate analytic semigroups, that is
\[ 
-\p_{\h_i}\Phi_i{i}\in \mathcal{H}(h^{4+\beta}(\s), h^{1+\alpha}(\s)), i=1,2,
\]
 while the elements on the secondary diagonal belong to $\kL(h^{2+\beta}(\s),h^{1+\beta}(\s)),$ that is they have lower order.
 We obtain then form \cite[Theorem 1.6.1]{Am} that the matrix $\p\Phi(0)$ is a generator, which  completes the proof.
\end{proof}

We give  now a short proof of our second main result, Theorem \ref{T:2}.
\begin{proof}[Proof of Theorem \ref{T:2}] Let $0<\beta<\alpha.$
Since $\Phi$ is analytic and $\p\Phi(0)$ generates a strongly continuous and analytic semigroup, we find an open neighborhood $\widetilde \cO$ of 
$0$ in $h^{4+\beta}(\s)$ such that $-\p\Phi(\h_1,\h_2)\in\mathcal{H}((h^{4+\beta}(\s))^2,(h^{1+\beta}(\s))^2)$ for all $(\h_1,\h_2)\in\widetilde\cO^2.$
Letting $\cO:=\widetilde\cO\cap h^{4+\alpha}(\s)),$ we find that $-\p\Phi(\h_1,\h_2)\in\mathcal{H}((h^{4+\alpha}(\s))^2,(h^{1+\alpha}(\s))^2)$ is, for all $(\h_1,\h_2)\in\cO^2,$
the realisation of the operator $-\p\Phi(\h_1,\h_2)\in\mathcal{H}((h^{4+\beta}(\s))^2,(h^{1+\beta}(\s))^2).$
Whence, the assumptions of \cite[Theorem 8.4.1]{L} are all satisfied and the desired assertion follows at once.
\end{proof}

\end{document}